\documentclass[12pt,thmsa]{article}
\usepackage{amssymb}
\usepackage{amsfonts}
\usepackage{amsmath}

\newtheorem{theorem}{Theorem}[section]

\newtheorem{corollary}[theorem]{Corollary}
\newtheorem{remark}{Remark}[section]
\newtheorem{notation}{Notation}[section]
\newtheorem{lemma}[theorem]{Lemma}

\newtheorem{definition}{Definition}[section]
\newenvironment{proof}[1][Proof]{\noindent\textbf{#1.} }{\ \rule{0.5em}{0.5em}}

\begin{document}

\title{On the indefinite Kirchhoff type problems with local sublinearity and
linearity\thanks{%
J. Sun was supported by the National Natural Science Foundation of China
(Grant No. 11201270, No.11271372), Shandong Natural Science Foundation
(Grant No. ZR2012AQ010), and Young Teacher Support Program of Shandong
University of Technology. T. F. Wu was supported in part by the National
Science Council and the National Center for Theoretical Sciences (South),
Taiwan.}}
\date{}
\author{Juntao Sun \\
School of Science\\
Shandong University of Technology, Zibo, 255049, PR. China \\
sunjuntao2008@163.com \and Yi-hsin Cheng and Tsung-fang Wu \\
Department of Applied Mathematics \\
National University of Kaohsiung, Kaohsiung 811, Taiwan \\
d0984103@mail.nuk.edu.tw; tfwu@nuk.edu.tw }
\maketitle

\begin{abstract}
The purpose of this paper is to study the indefinite Kirchhoff type problem:
\begin{equation*}
\left\{
\begin{array}{ll}
M\left( \int_{\mathbb{R}^{N}}(|\nabla u|^{2}+u^{2})dx\right) \left[ -\Delta
u+u\right] =f(x,u) & \text{in }\mathbb{R}^{N}, \\
0\leq u\in H^{1}\left( \mathbb{R}^{N}\right) , &
\end{array}%
\right.
\end{equation*}%
where $N\geq 1$, $M(t)=am\left( t\right) +b$, $m\in C(\mathbb{R}^{+})$ and $%
f(x,u)=g(x,u)+h(x)u^{q-1}$. We require that $f$ is \textquotedblleft
local\textquotedblright\ sublinear at the origin and \textquotedblleft
local\textquotedblright\ linear at infinite. Using the mountain pass theorem
and Ekeland variational principle, the existence and multiplicity of
nontrivial solutions are obtained. In particular, the criterion of existence
of three nontrivial solutions is established.
\end{abstract}

\section{Introduction}

In this paper, we investigate the existence and multiplicity of nontrivial
solutions for a Kirchhoff type problem:
\begin{equation}
\left\{
\begin{array}{ll}
M\left( \int_{\mathbb{R}^{N}}|\nabla u|^{2}+u^{2})dx\right) \left[ -\Delta
u+u\right] =f(x,u) & \text{in }\mathbb{R}^{N}, \\
0\leq u\in H^{1}\left( \mathbb{R}^{N}\right) , &
\end{array}%
\right.  \tag*{$\left( K\right) $}
\end{equation}%
where $N\geq 1,$ $f\in C(\mathbb{R}^{N}\times \mathbb{R}^{+},\mathbb{R})$
and $M:\mathbb{R}\rightarrow \mathbb{R}$ is a given function whose
properties will be given later.

Problem $\left( K\right)$ is related to the stationary analogue of the
equation
\begin{equation*}
\rho\frac{\partial^{2}u}{\partial t^{2}}-\left( \frac{P_{0}}{h}+\frac{E}{2L}%
\int_{0}^{L}\left|\frac{\partial u}{\partial x}\right|^{2}dx\right) \frac{%
\partial^{2}u}{\partial x^{2}}=0,
\end{equation*}%
presented by Kirchhoff \cite{K} in 1883. This equation is an extension of
the classical d'Alembert's wave equation by considering the effects of the
changes in the length of the string during the vibrations. Such problems are
often referred to as being nonlocal because of the presence of the integral.
When $M\left( t\right) =at+b$ $\left( a,b>0\right) ,$ it is degenerate if $%
b=0$ and nondegenerate otherwise.

After Lions \cite{L} introduced an abstract framework to the Kirchhoff type
problem, Problem $\left( K\right)$ began to receive much attention. Most
researchers studied the Kirchhoff type problems on bounded domain $\Omega
\subset \mathbb{R}^{N}$ with the following version
\begin{equation}
\left\{
\begin{array}{ll}
-M\left( \int_{\Omega }|\nabla u|^{2}dx\right) \Delta u=f(x,u) & \text{ in }%
\Omega , \\
u=0 & \ \text{on }\partial \Omega.%
\end{array}%
\right.  \label{3}
\end{equation}%
For example, Bensedik and Bouchekif \cite{BB}, Chen et al. \cite{CKW}, Alves
et al. \cite{ACM} and Ma and Rivera \cite{MR}, using variational methods,
proved the existence and multiplicity of positive solutions while Zhang and
Perera \cite{ZP} obtained sign changing solutions via invariant sets of
descent flow. In particular, Alves et al. \cite{ACM} studied the conditions
of $M$ and $f$ that permit the existence of a positive solution and
concluded that this is possible if $M$ does not grow too fast in a suitable
interval near zero with $f$ being locally Lipschitz subject to some
prescribed criteria. Bensedik and Bouchekif \cite{BB} studied the
asymptotically linear case and obtained the existence of positive solutions
of Problem $\left( \ref{3}\right) $ when the function $M$ is a
non-decreasing function and $M\geq m_{0}$ for some $m_{0}>0$, and the
assumptions about the asymptotic behaviors of $f$ near zero and infinite are
the following

\begin{itemize}
\item[$\left( f_{1}\right) $] $t\longmapsto \frac{f\left( x,t\right) }{t}$
is a non-decreasing function for any fixed $x\in \overline{\Omega };$

\item[$\left( f_{2}\right) $] $\lim_{t\rightarrow 0}\frac{f\left( x,t\right)
}{t}=\overline{p}\left( x\right) $ and $\lim_{t\rightarrow \infty }\frac{%
f\left( x,t\right) }{t}=\overline{q}\left( x\right) $ uniformly in $x\in
\Omega ,$ where $0\leq \overline{p}\left( x\right) ,\overline{q}\left(
x\right) \in L^{\infty }\left( \Omega \right) $ and $\sup_{x\in \Omega }%
\overline{p}\left( x\right) <m_{0}\lambda _{1},$ $\lambda _{1}$ is the first
eigenvalue of $\left( -\Delta ,H_{0}^{1}\left( \Omega \right) \right) .$
\end{itemize}

Compared with the case of the bounded domain $\Omega \subset \mathbb{R}^{N}$%
, the case of the whole space $\mathbb{R}^{N}$ has been considered by a few
authors, see \cite{ADP, CL2, HZ, JW, LLS, SW, W, WTXZ, YT}, and the
references therein. More precisely, Li et al. \cite{LLS} considered the
following Kirchhoff type problem:%
\begin{equation}
\begin{array}{ll}
\left( a+\lambda \int_{\mathbb{R}^{N}}\left( |\nabla u|^{2}+bu^{2}\right)
dx\right) \left[ -\Delta u+bu\right] =f(u) & \text{in }\mathbb{R}^{N},%
\end{array}
\label{4}
\end{equation}%
where $N\geq 3$, $a$ and $b$ are positive constants, and $\lambda \geq 0$ is
a parameter. Under the weaker assumption $\lim_{t\rightarrow \infty }\frac{%
f(t)}{t}=\infty $, a positive radial solution of Equation $\left( \ref{4}%
\right) $ was constructed by applying a monotonicity trick of Jeanjean \cite%
{J} whenever $\lambda \geq 0$ small enough. He and Zou \cite{HZ} studied the
multiplicity and concentration behavior of positive solutions for the
following Kirchhoff type problem:%
\begin{equation}
\left\{
\begin{array}{ll}
-\left( \varepsilon ^{2}a+\varepsilon b\int_{\mathbb{R}^{N}}|\nabla
u|dx\right) \Delta u+V\left( x\right) u=f(u) & \text{in }\mathbb{R}^{N}, \\
0<u\in H^{1}\left( \mathbb{R}^{N}\right) , &
\end{array}%
\right.  \label{5}
\end{equation}%
where $\varepsilon >0$ is a parameter, $a,b>0$ are constants, and $f$ is a
continuous superlinear and subcritical nonlinear term. When $V$ has at least
one minimum, the authors proved that Equation $\left( \ref{5}\right) $ has a
ground state solution for $\varepsilon >0$ sufficiently small. Moreover,
they investigated the relation between the number of positive solutions and
the topology of the set of the global minima of the potentials by using
minimax theorems together with the Ljusternik-Schnirelmann theory.

Inspired by the above facts, the aim of this paper is to consider the
indefinite Kirchhoff type equations with local sublinearity and linearity.
To the author's knowledge, this case seems to be considered by few authors.
We mainly study the existence and multiplicity of nontrivial solutions for
Problem $(K)$. Furthermore, the non-existence of nontrivial solutions are
also discussed. In this paper, we consider the following Kirchhoff type
problem:%
\begin{equation}
\left\{
\begin{array}{ll}
M\left( \int_{\mathbb{R}^{N}}\left( |\nabla u|^{2}+u^{2}\right) dx\right) %
\left[ -\Delta u+u\right] =g(x,u)+h\left( x\right) u^{q-1} & \text{in }%
\mathbb{R}^{N}, \\
0\leq u\in H^{1}\left( \mathbb{R}^{N}\right) , &
\end{array}%
\right.  \tag*{$\left( K_{a,h}\right) $}
\end{equation}%
where $1<q<2,h\in L^{2/\left( 2-q\right) }\left( \mathbb{R}^{N}\right) ,\
M(t)=am\left( t\right) +b,$ the parameters $a,b>0$ and $m$ is a continuous
function on $\mathbb{R}^{+}$ such that $m\left( t\right) \geq 0$ for all $%
t>0.$ We assume that the function $g$ satisfies the following conditions:

\begin{itemize}
\item[$\left( D_{1}\right) $] $g\left( x,s\right) $ is a continuous function
on $\mathbb{R}^{N}\times \mathbb{R}$ such that $g(x,s)\equiv 0$ for all $s<0$
and $x\in \mathbb{R}^{N}.$ Moreover, there exists $p_{1}\in L^{\infty
}\left( \mathbb{R}^{N}\right) $ with $p_{1}^{+}\not\equiv 0$ such that
\begin{equation*}
\frac{g\left( x,s\right) }{s}\geq p_{1}\text{ for all }s>0\text{ and }x\in
\mathbb{R}^{N}
\end{equation*}%
and%
\begin{equation*}
\lim_{s\rightarrow 0^{+}}\frac{g\left( x,s\right) }{s}=p_{1}\text{ uniformly
for }x\in \mathbb{R}^{N},
\end{equation*}
where $p_{1}^{+}=\sup \left\{ p_{1},0\right\} ;$

\item[$\left( D_{2}\right) $] there exists $p_{2}\in L^{\infty }\left(
\mathbb{R}^{N}\right) $ with $p_{2}^{+}\not\equiv 0$ such that $%
\lim_{s\rightarrow \infty }\frac{g\left( x,s\right) }{s}=p_{2}\left(
x\right) $ uniformly for $x\in \mathbb{R}^{N},$ where $p_{2}^{+}=\sup
\left\{ p_{2},0\right\} ;$

\item[$\left( D_{3}\right) $] $\left\vert p_{1}^{+}\right\vert _{\infty }<b<%
\frac{1}{\mu ^{\ast }},$ where
\begin{equation}
\mu ^{\ast }:=\inf \left\{ \int_{\mathbb{R}^{N}}(|\nabla
u|^{2}+u^{2})dx:u\in H^{1}(\mathbb{R}^{N}),\int_{\mathbb{R}%
^{N}}p_{2}(x)u^{2}dx=1\right\} ;  \label{2}
\end{equation}

\item[$\left( D_{4}\right) $] there exists $R_{0}>0$ such that
\begin{equation*}
\sup \{\frac{g(x,s)}{s}:s>0\}\leq \min \left\{ 1,b\right\} \text{ uniformly
on }\left\vert x\right\vert \geq R_{0}.
\end{equation*}
\end{itemize}

\begin{remark}
By conditions $\left( D_{1}\right) $ and $\left( D_{2}\right) ,$ the
nonlinear term $f\left( x,s\right) :=g(x,s)+h\left( x\right) s^{q-1}$ for $%
s>0$ is \textquotedblleft local\textquotedblright\ sublinear at the origin
and \textquotedblleft local\textquotedblright\ linear at infinite, i.e.%
\begin{equation*}
\lim_{s\rightarrow 0^{+}}\frac{f\left( x,s\right) }{s^{q-1}}=h^{+}\left(
x\right) \text{ uniformly for }x\in \Omega _{h}^{+}:=\left\{ x\in \mathbb{R}%
^{N}:h\left( x\right) >0\right\}
\end{equation*}%
and%
\begin{equation*}
\lim_{s\rightarrow \infty }\frac{f\left( x,s\right) }{s}=p_{2}^{+}\left(
x\right) \text{ uniformly for }x\in \Omega _{p_{2}}^{+}:=\left\{ x\in
\mathbb{R}^{N}:p_{2}\left( x\right) >0\right\} ,
\end{equation*}%
where $p^{+}=\sup \left\{ p,0\right\} .$
\end{remark}

It is well known that Equation $(K_{a,h})$ is variational and its solutions
are the critical points of the functional defined in $H^{1}(\mathbb{R}^{N})$
by
\begin{equation*}
I_{a,h}(u)=\frac{a}{2}\widehat{m}\left( \Vert u\Vert ^{2}\right) +\frac{b}{2}%
\Vert u\Vert ^{2}-\int_{\mathbb{R}^{N}}G(x,u)dx-\int_{\mathbb{R}%
^{N}}h\left\vert u^{+}\right\vert ^{q}dx,
\end{equation*}%
where $\widehat{m}\left( t\right) =\int_{0}^{t}m\left( s\right) ds$, $%
\left\Vert u\right\Vert =\left( \int_{\mathbb{R}^{N}}\left( |\nabla
u|^{2}+u^{2}\right) dx\right) ^{1/2}$ is a standard norm in $H^{1}\left(
\mathbb{R}^{N}\right) $, $G(x,u)=\int_{0}^{u}g(x,s)ds$ and $u^{+}=\sup
\left\{ u,0\right\} .$ Furthermore, it is easy to prove that the functional $%
I_{a,h}$ is of class $C^{1}$ in $H^{1}\left( \mathbb{R}^{N}\right) $, and
that
\begin{eqnarray*}
\langle I_{a,h}^{\prime }(u),v\rangle &=&\left[ am\left( \Vert u\Vert
^{2}\right) +b\right] \int_{\mathbb{R}^{N}}(\nabla u\cdot \nabla
v+uv)dx-\int_{\mathbb{R}^{N}}g(x,u)vdx \\
&&-\int_{\mathbb{R}^{N}}h\left\vert u^{+}\right\vert ^{q-2}u^{+}vdx.
\end{eqnarray*}%
Hence if $u\in H^{1}(\mathbb{R}^{N})$ is a nonzero critical point of $%
I_{a,h} $, then $u$ is a nontrivial solution of Equation $(K_{a,h}).$

Before stating our result we need to introduce some notations and
definitions.

\begin{notation}
Throughout this paper, we denote by $\left\vert \cdot \right\vert _{r}$ the $%
L^{r}$-norm, $2\leq r\leq \infty $ and $B_{r}:=\left\{ u\in H^{1}\left(
\mathbb{R}^{N}\right) :\left\Vert u\right\Vert <r\right\} $ is an open ball
in $H^{1}\left( \mathbb{R}^{N}\right) .$ The letter $C$ will denote various
positive constants whose value may change from line to line but are not
essential to the analysis of the problem. Also if we take a subsequence of a
sequence $\left\{ u_{n}\right\} $ we shall denote it again $\left\{
u_{n}\right\} .$ We use $o\left( 1\right) $ to denote any quantity which
tends to zero when $n\rightarrow \infty .$
\end{notation}

\begin{definition}
$u$ is a ground state of Equation $(K_{a,h})$ we mean that $u$ is such a
solution of Equation $(K_{a,h})$ which has the least energy among all
nontrivial solutions of Equation $(K_{a,h}).$
\end{definition}

We also need the following assumptions:

\begin{itemize}
\item[$\left( D_{5}\right) $] $m\left( t\right) \rightarrow +\infty $ as $%
t\rightarrow \infty ;$

\item[$\left( D_{6}\right) $] there exist $\delta _{0},d_{0}>0$ such that $%
m\left( t\right) \geq d_{0}t^{^{\delta _{0}}}$ for all $t>0.$
\end{itemize}

Now, we give our main results.

\begin{theorem}
\label{t1-1}$\left( i\right) $ Suppose that conditions ${(D_{1})}-{(D_{5})}$
hold. If $h\equiv 0,$ then there exists $a^{\ast }>0$ such that for every $%
a\in \left( 0,a^{\ast }\right) ,$ Equation $\left( K_{a,h}\right) $ has one
nontrivial solution $u_{0}^{+}$ with $I_{a,h}\left( u_{0}^{+}\right) >0.$%
\newline
$\left( ii\right) $ Suppose that conditions ${(D_{1})}-{(D_{4})}$ and $%
\left( D_{6}\right) $ hold. If $h\equiv 0,$ then there exists $a^{\ast \ast
}>0$ such that for every $a\in \left( 0,a^{\ast \ast }\right) ,$ Equation $%
\left( K_{a,h}\right) $ has two nontrivial solutions $u_{0}^{-}$ and $%
u_{0}^{+}$ with $I_{a,h}\left( u_{0}^{-}\right) <0<I_{a,h}\left(
u_{0}^{+}\right) ,$ and $u_{0}^{-}$ is a ground state solution.
\end{theorem}

\begin{theorem}
\label{t1-2}$\left( i\right) $ Suppose that conditions ${(D_{1})}-{(D_{4})}$
hold. Then there exists $\Lambda _{0}>0$ such that for every $a>0$ and $h\in
L^{2/\left( 2-q\right) }\left( \mathbb{R}^{N}\right) $ with $0<\left\vert
h^{+}\right\vert _{L^{2/\left( 2-q\right) }}<\Lambda _{0},$ Equation $\left(
K_{a,h}\right) $ has one nontrivial solution $u_{h,1}^{-}$ with $%
I_{a,h}\left( u_{h,1}^{-}\right) <0.$\newline
$\left( ii\right) $ Suppose that conditions ${(D_{1})}-{(D_{5})}$ hold. Then
there exist $a^{\ast },\Lambda _{0}>0$ such that for every $a\in \left(
0,a^{\ast }\right) $ and $h\in L^{2/\left( 2-q\right) }\left( \mathbb{R}%
^{N}\right) $ with $0<\left\vert h^{+}\right\vert _{L^{2/\left( 2-q\right)
}}<\Lambda _{0},$ Equation $\left( K_{a,h}\right) $ has two nontrivial
solutions $u_{h,1}^{-}$ and $u_{h}^{+}$ with $I_{a,h}\left(
u_{h,1}^{-}\right) <0<I_{a,h}\left( u_{h}^{+}\right) .$\newline
$\left( iii\right) $ Suppose that conditions ${(D_{1})}-{(D_{4})}$ and $%
\left( D_{6}\right) $ hold. Then there exist $\overline{a}_{0},\overline{%
\Lambda }_{0}>0$ such that for every $a\in \left( 0,\overline{a}_{0}\right) $
and $h\in L^{2/\left( 2-q\right) }\left( \mathbb{R}^{N}\right) $ with $h\geq
0$ and $0<\left\vert h\right\vert _{L^{2/\left( 2-q\right) }}<\overline{%
\Lambda }_{0},$ Equation $\left( K_{a,h}\right) $ has three nontrivial
solutions $u_{h,1}^{-},u_{h,2}^{-}$ and $u_{h}^{+}$ with
\begin{equation*}
I_{a,h}\left( u_{h,2}^{-}\right) <I_{a,h}\left( u_{h,1}^{-}\right)
<0<I_{a,h}\left( u_{h}^{+}\right) ,
\end{equation*}%
and $u_{h,2}^{-}$ is a ground state solution.
\end{theorem}

We now turn to example $m\left( t\right) =t$ for $t\geq 0.$

\begin{corollary}
Suppose that conditions ${(D_{1})}-{(D_{4})}$ hold and $m\left( t\right) =t$
for $t\geq 0.$ Then there exist $\overline{a}_{0},\overline{\Lambda }_{0}>0$
such that for every $a\in \left( 0,\overline{a}_{0}\right) $ and $h\in
L^{2/\left( 2-q\right) }\left( \mathbb{R}^{N}\right) $ with $h\geq 0$ and $%
0<\left\vert h\right\vert _{L^{2/\left( 2-q\right) }}<\overline{\Lambda }%
_{0},$ Equation $\left( K_{a,h}\right) $ has three nontrivial solutions $%
u_{h,1}^{-},u_{h,2}^{-}$ and $u_{h}^{+}$ with
\begin{equation*}
I_{a,h}\left( u_{h,2}^{-}\right) <I_{a,h}\left( u_{h,1}^{-}\right)
<0<I_{a,h}\left( u_{h}^{+}\right) ,
\end{equation*}%
and $u_{h,2}^{-}$ is a ground state solution.
\end{corollary}

\begin{remark}
If the functions $p_{1},p_{2}$ and $h$ are nonnegative, then by the maximum
principle, all solutions of Theorems \ref{t1-1} and \ref{t1-2} are positive
ones of Equation $\left( K_{a,h}\right) .$
\end{remark}

On the non-existence of nontrivial solutions we have the following result.

\begin{theorem}
\label{t1-3}Suppose in addition to the condition ${(D_{2})}$ holds and $h\in
L^{2/\left( 2-q\right) }\left( \mathbb{R}^{N}\right) $, we also have\newline
$\left( D_{7}\right) \ s\longmapsto \frac{g\left( x,s\right) }{s}$ is
non-decreasing function for any fixed $x\in \mathbb{R}^{N};$\newline
$\left( D_{8}\right) $ $m\left( t\right) >\frac{1-b\mu ^{\ast }}{a\mu ^{\ast
}}+\frac{\left\vert h^{+}\right\vert _{L^{2/\left( 2-q\right) }}}{aS_{2}^{q}}%
t^{\left( q-2\right) /2}$ for all $t>0,$ where $\mu ^{\ast }>0$ is defined
in (\ref{2}).\newline
Then Equation $\left( K_{a,h}\right) $ does not admits any nontrivial
solution.
\end{theorem}

The remainder of this paper is organized as follows. In Section 2, some
preliminary results are presented. In Section 3 and 4, we give the proofs of
Theorems \ref{t1-1} and \ref{t1-2}. In Section 5, we give the proof of
Theorem \ref{t1-3}.

\section{Preliminaries}

Throughout this paper, we denote by $S_{r}$ the best Sobolev constant for
the imbedding of $H^{1}\left( \mathbb{R}^{N}\right) $ in $L^{r}\left(
\mathbb{R}^{N}\right) $ with $2\leq r<2^{\ast }.$ In particular,
\begin{equation*}
\left\vert u\right\vert _{r}\leq S_{r}^{-1}\left\Vert u\right\Vert \text{
for all }u\in H^{1}\left( \mathbb{R}^{N}\right) \backslash \left\{ 0\right\}
.
\end{equation*}

Next, we give a useful theorem. It is the variant version of the mountain
pass theorem, which allows us to find a so-called Cerami type $(PS)$
sequence. The properties of this kind of $(PS)$ sequence are very helpful in
showing the boundedness of the sequence in the asymptotically linear case.

\begin{theorem}
\label{t2}(\cite{E2}, Mountain Pass Theorem). Let $E$ be a real Banach space
with its dual space $E^{\ast },$ and suppose that $I\in C^{1}(E,\mathbb{R})$
satisfies
\begin{equation*}
\max \{I(0),I(e)\}\leq \mu <\eta \leq \inf_{\Vert u\Vert =\rho }I(u),
\end{equation*}%
for some $\mu <\eta ,\rho >0$ and $e\in E$ with $\Vert e\Vert >\rho .$ Let $%
c\geq \eta $ be characterized by
\begin{equation*}
\alpha =\inf_{\gamma \in \Gamma }\max_{0\leq \tau \leq 1}I(\gamma (\tau )),
\end{equation*}%
where $\Gamma =\{\gamma \in C([0,1],E):\gamma (0)=0,\gamma (1)=e\}$ is the
set of continuous paths joining $0$ and $e$, then there exists a sequence $%
\{u_{n}\}\subset E$ such that
\begin{equation*}
I(u_{n})\rightarrow \alpha \geq \eta \quad \text{and}\quad (1+\Vert
u_{n}\Vert )\Vert I^{\prime }(u_{n})\Vert _{E^{\ast }}\rightarrow 0,\quad
\text{as}\ n\rightarrow \infty .
\end{equation*}
\end{theorem}

In what follows, we give the following Lemmas which ensure that the
functional $I_{a,h}$ has the mountain pass geometry.

\begin{lemma}
\label{m1}Let $1<q<2<r,A>0,B>0$, and consider the function
\begin{equation*}
\Psi _{A,B}\left( t\right) :=t^{2}-At^{q}-Bt^{r}\text{ for }t\geq 0.
\end{equation*}%
Then $\max_{t\geq 0}\Psi _{A,B}\left( t\right) >0$ if and only if
\begin{equation*}
A^{r-2}B^{2-q}<d\left( r,q\right) :=\frac{\left( r-2\right) ^{r-2}\left(
2-q\right) ^{2-q}}{\left( r-q\right) ^{r-q}}.
\end{equation*}%
Furthermore, for $t=t_{B}:=\left[ \left( 2-q\right) /B\left( r-q\right) %
\right] ^{1/\left( r-2\right) }$, one has
\begin{equation*}
\Psi _{A,B}\left( t_{B}\right) =t_{B}^{2}\left[ \frac{r-2}{r-q}-AB^{\frac{2-q%
}{r-2}}\left( \frac{r-q}{2-q}\right) ^{\frac{2-q}{r-2}}\right] >0.
\end{equation*}
\end{lemma}

\begin{proof}
The proof is essentially the same as that in \cite[Lemma 3.2]{FGU}, and we
omit it here.
\end{proof}

\begin{lemma}
\label{m2}Let $1<q<2<r<k,\overline{A}>0,\overline{B}>0$, and consider the
function
\begin{equation*}
\Phi _{\overline{A},\overline{B}}\left( t\right) :=t^{k}-\overline{A}t^{q}-%
\overline{B}t^{r}\text{ for }t\geq 0.
\end{equation*}%
Then for $t=t_{\overline{B}}:=\left[ \overline{B}\left( r-q\right) /\left(
k-q\right) \right] ^{1/\left( k-r\right) }$, one has
\begin{equation}
\Phi _{\overline{A},\overline{B}}\left( t_{\overline{B}}\right) =-t_{%
\overline{B}}^{q}\left[ \left( \frac{\overline{B}\left( r-q\right) }{k-q}%
\right) ^{\left( k-q\right) /\left( k-r\right) }\left( \frac{k-r}{r-q}%
\right) +\overline{A}\right] <0.  \label{3.1}
\end{equation}%
Furthermore, there exist $t_{0},t_{1}>0$ such that $\min_{t\geq 0}\Phi _{%
\overline{A},\overline{B}}\left( t\right) =\Phi _{\overline{A},\overline{B}%
}\left( t_{0}\right) <0$ and $\Phi _{\overline{A},\overline{B}}\left(
t\right) \geq 0$ for all $t\geq t_{1}.$
\end{lemma}

\begin{proof}
Since $\Phi _{\overline{A},\overline{B}}\left( t\right) =t^{q}\left( t^{k-q}-%
\overline{A}-\overline{B}t^{r-q}\right) $, it follows that $\Phi _{\overline{%
A},\overline{B}}\left( t\right) <0$ if and only if $t^{k-q}-\overline{A}-%
\overline{B}t^{r-q}<0$. The derivative of $t^{k-q}-\overline{A}-\overline{B}%
t^{r-q}$ vanishes exactly for $t=t_{\overline{B}}$ and one readily computes $%
\Phi _{\overline{A},\overline{B}}\left( t_{\overline{B}}\right) $, as
indicated in $\left( \ref{3.1}\right) $. The conclusion of the lemma then
follows easily.
\end{proof}

\begin{lemma}
\label{m3}Let $2<r<k,A_{0}>0,B_{0}>0$, and consider the function
\begin{equation*}
\Theta _{A_{0},B_{0}}\left( t\right) :=t^{k}+A_{0}t^{2}-B_{0}t^{r}\text{ for
}t\geq 0.
\end{equation*}%
Then $\min_{t\geq 0}\Theta _{A_{0},B_{0}}\left( t\right) <0$ if and only if
\begin{equation*}
A_{0}^{k-r}B_{0}^{2-k}<d_{0}\left( r\right) :=\frac{\left( k-r\right)
^{k-r}\left( r-2\right) ^{r-2}}{\left( k-2\right) ^{k-2}}.
\end{equation*}%
For $t=t_{B_{0}}:=\left[ B_{0}\left( r-2\right) /\left( k-2\right) \right]
^{1/\left( k-r\right) }$, one has%
\begin{equation}
\Theta _{A_{0},B_{0}}\left( t_{B_{0}}\right) =t_{B_{0}}^{2}\left[
A_{0}-B_{0}\left( \frac{B_{0}\left( r-2\right) }{k-2}\right) ^{\frac{r-2}{k-r%
}}\left( \frac{k-r}{k-2}\right) \right] <0.  \label{3.3}
\end{equation}%
Furthermore, there exist $t_{0}<t_{B_{0}}<t_{1}$ such that $\Theta
_{A_{0},B_{0}}\left( t_{0}\right) =\Theta _{A_{0},B_{0}}\left( t_{1}\right)
=0$ and $\Theta _{A_{0},B_{0}}\left( t\right) >0$ for all $t\in \left(
0,t_{0}\right) \cup \left( t_{1},\infty \right) .$
\end{lemma}

\begin{proof}
Since $\Theta _{A_{0},B_{0}}\left( t\right) =t^{2}\left(
t^{k-2}+A_{0}-B_{0}t^{r-2}\right) $, it follows that $\Theta
_{A_{0},B_{0}}\left( t\right) <0$ if and only if $%
t^{k-2}+A_{0}-B_{0}t^{r-2}<0.$ The derivative of $t^{k-2}+A_{0}-B_{0}t^{r-2}$
vanishes exactly for $t=t_{B_{0}}$ and one readily computes $\Theta
_{A_{0},B_{0}}\left( t_{B_{0}}\right) $, as indicated in $\left( \ref{3.3}%
\right) $. The conclusion of the lemma then follows easily.
\end{proof}

\begin{lemma}
\label{lem1}Suppose that conditions ${(D_{1})}-\left( D_{3}\right) $ hold.
Then there exist $\Lambda _{0},\rho >0$ such that for every $h\in
L^{2/\left( 2-q\right) }\left( \mathbb{R}^{N}\right) $ with $\left\vert
h^{+}\right\vert _{L^{2/\left( 2-q\right) }}<\Lambda _{0},$
\begin{equation*}
\inf \{I_{a,h}(u):u\in H^{1}(\mathbb{R}^{N})\ \text{with}\ \Vert u\Vert
=\rho \}>\eta
\end{equation*}%
for some $\eta >0.$ Furthermore, if $h^{+}\equiv 0,$ then there exists $\rho
_{0}>\rho $ such that $I_{a,h}\left( u\right) >0$ for all $u\in B_{\rho
_{0}}\backslash \left\{ 0\right\} $ and $\inf_{u\in B_{\rho
_{0}}}I_{a,h}\left( u\right) =0.$
\end{lemma}

\begin{proof}
By conditions ${(D_{1})}-\left( D_{3}\right) $ and $\left( D_{5}\right) ,$
and noticing that $\lim_{s\rightarrow +\infty }\frac{g\left( x,s\right) }{%
s^{r-1}}=0$ uniformly in $x\in \mathbb{R}^{N}$ for any fixed $2<r<2^{\ast }$
($2^{\ast }=\infty $ if $N=1,2$ and $2^{\ast }=\frac{2N}{N-2}$ if $N\geq 3$%
), it is easy to see that for every $\epsilon >0$, there exists $C_{\epsilon
}=C\left( \epsilon ,r,g\right) >0$ such that
\begin{equation}
g(x,s)\leq \frac{\left\vert p_{1}^{+}\right\vert _{\infty }+\epsilon }{2}s+%
\frac{C_{\epsilon }}{r}|s|^{r-1},\quad \text{for all}\ s\geq 0  \label{6}
\end{equation}%
and
\begin{equation*}
G(x,s)\leq \frac{\left\vert p_{1}^{+}\right\vert _{\infty }+\epsilon }{2}%
s^{2}+\frac{C_{\epsilon }}{r}|s|^{r},\quad \text{for all}\ s\geq 0.
\end{equation*}%
Since $\left\vert p_{1}^{+}\right\vert _{\infty }<b,$ we can find $\epsilon
_{0}>0$ with $\left\vert p_{1}^{+}\right\vert _{\infty }+\epsilon _{0}<b$
and there is a $C_{0}=C\left( \epsilon _{0},r,g\right) >0$ such that
\begin{equation}
G(x,s)\leq \frac{\left\vert p_{1}^{+}\right\vert _{\infty }+\epsilon _{0}}{2}%
s^{2}+\frac{C_{0}}{r}|s|^{r},\quad \text{for all}\ s\geq 0.  \label{3.2}
\end{equation}%
Thus, from $(\ref{3.2})$ and the Sobolev inequality, we have for all $u\in
H^{1}(\mathbb{R}^{N})$,
\begin{eqnarray}
\int_{\mathbb{R}^{N}}G(x,u)dx &\leq &\frac{\left\vert p_{1}^{+}\right\vert
_{\infty }+\epsilon _{0}}{2}\int_{\mathbb{R}^{N}}u^{2}dx+\frac{C_{0}}{r}%
\int_{\mathbb{R}^{N}}|u|^{r}dx  \notag \\
&\leq &\frac{\left\vert p_{1}^{+}\right\vert _{\infty }+\epsilon _{0}}{2}%
\Vert u\Vert ^{2}+\frac{C_{0}S_{r}^{-r}}{r}\Vert u\Vert ^{r},  \label{3.4}
\end{eqnarray}%
which implies that
\begin{eqnarray}
I_{a,h}(u) &=&\frac{a}{2}\widehat{m}\left( \Vert u\Vert ^{2}\right) +\frac{b%
}{2}\Vert u\Vert ^{2}-\int_{\mathbb{R}^{N}}G(x,u)dx-\frac{1}{q}\int_{\mathbb{%
R}^{N}}h\left\vert u^{+}\right\vert ^{q}dx  \notag \\
&\geq &\frac{b}{2}\Vert u\Vert ^{2}-\frac{\left\vert p_{1}^{+}\right\vert
_{\infty }+\epsilon _{0}}{2}\Vert u\Vert ^{2}-\frac{C_{0}S_{r}^{-r}}{r}\Vert
u\Vert ^{r}-\frac{\left\vert h^{+}\right\vert _{L^{2/\left( 2-q\right) }}}{%
qS_{2}^{q}}\left\Vert u\right\Vert ^{q}  \notag \\
&\geq &\frac{b-\left\vert p_{1}^{+}\right\vert _{\infty }-\epsilon _{0}}{2}%
\Vert u\Vert ^{2}-\frac{C_{0}}{rS_{r}^{r}}\Vert u\Vert ^{r}-\frac{\left\vert
h^{+}\right\vert _{L^{2/\left( 2-q\right) }}}{qS_{2}^{q}}\left\Vert
u\right\Vert ^{q}  \label{3.7}
\end{eqnarray}%
for all $u\in H^{1}(\mathbb{R}^{N}).$ For $h\in L^{2/\left( 2-q\right)
}\left( \mathbb{R}^{N}\right) $ with $\left\vert h^{+}\right\vert
_{L^{2/\left( 2-q\right) }}>0.$ We now apply to Lemma \ref{m1} above with%
\begin{equation*}
A=\frac{2\left\vert h^{+}\right\vert _{L^{2/\left( 2-q\right) }}}{%
qS_{2}^{q}\left( b-\left\vert p_{1}^{+}\right\vert _{\infty }-\epsilon
_{0}\right) }>0\text{ and }B=\frac{2C_{0}}{rS_{r}^{r}\left( b-\left\vert
p_{1}^{+}\right\vert _{\infty }-\epsilon _{0}\right) }>0.
\end{equation*}%
This shows that for all $u\in H^{1}(\mathbb{R}^{N})$ with $\left\Vert
u\right\Vert =t_{B}=\left[ \left( 2-q\right) /B\left( r-q\right) \right]
^{1/\left( r-2\right) }$,
\begin{equation*}
I_{a,h}\left( u\right) \geq \frac{b-\left\vert p_{1}^{+}\right\vert _{\infty
}-\epsilon _{0}}{2}\Psi _{A,B}\left( t_{B}\right) >0
\end{equation*}%
provided that $A^{r-2}B^{2-q}<d\left( r,s\right) $, i.e., provided that%
\begin{equation*}
\left\vert h^{+}\right\vert _{L^{2/\left( 2-q\right) }}<\Lambda _{0}:=\frac{%
\left( r-2\right) S_{2}^{q}}{2}\left( \frac{b-\left\vert
p_{1}^{+}\right\vert _{\infty }-\epsilon _{0}}{r-s}\right) ^{\left(
r-q\right) /\left( r-2\right) }\left( \frac{rS_{r}^{r}\left( 2-s\right) }{%
2C_{0}}\right) ^{\left( 2-q\right) /\left( r-2\right) }.
\end{equation*}%
Letting
\begin{equation*}
\rho =t_{B}=\left[ \left( 2-q\right) /B\left( r-q\right) \right] ^{1/\left(
r-2\right) }>0
\end{equation*}%
and
\begin{equation*}
\eta =\frac{b-\left\vert p_{1}^{+}\right\vert _{\infty }-\epsilon _{0}}{2}%
\Psi _{A,B}\left( t_{B}\right) >0,
\end{equation*}%
it is easy to see that the result holds. Moreover, if $h^{+}\equiv 0,$ then
by $\left( \ref{3.7}\right) ,$
\begin{equation*}
I_{a,h}\left( u\right) >0\text{ for all }u\in B_{\rho _{0}}\backslash
\left\{ 0\right\}
\end{equation*}%
and
\begin{equation*}
\inf_{u\in B_{\rho _{0}}}I_{a,h}\left( u\right) =0,
\end{equation*}%
where $\rho _{0}=\left( \frac{r-q}{2-r}\right) ^{1/\left( r-2\right) }\rho
>\rho .$ This completes the proof.
\end{proof}

\begin{lemma}
\label{lem2}Suppose that conditions ${(D_{1})}-\left( D_{3}\right) $ hold.
Then for each $h\in L^{2/\left( 2-q\right) }\left( \mathbb{R}^{N}\right) $
there exist $a^{\ast }>0$ and $e\in H^{1}(\mathbb{R}^{N})$ with $\Vert
e\Vert >\rho $ such that $I_{a,h}(e)<0$ for all $a\in \left( 0,a^{\ast
}\right) $, where $\rho $ is given by Lemma \ref{lem1}.
\end{lemma}

\begin{proof}
By the condition $\left( D_{3}\right) $, in view of the definition of $\mu
^{\ast }$ and $b<1/\mu ^{\ast }$, there is $\phi \in H^{1}(\mathbb{R}%
^{N})\backslash \left\{ 0\right\} $ with $\phi \geq 0$ such that $\int_{%
\mathbb{R}^{N}}p_{2}(x)\phi ^{2}dx=1$ and $b\mu ^{\ast }\leq b\Vert \phi
\Vert ^{2}<1$. According to the condition $\left( D_{2}\right) $ and Fatou's
lemma, we have%
\begin{eqnarray*}
\lim_{t\rightarrow +\infty }\frac{I_{0,h}(t\phi )}{t^{2}} &=&\frac{b}{2}%
\Vert \phi \Vert ^{2}-\lim_{t\rightarrow +\infty }\int_{\mathbb{R}^{N}}\frac{%
G(x,t\phi )}{t^{2}\phi ^{2}}\phi ^{2}dx-\lim_{t\rightarrow \infty }\frac{1}{%
t^{2-q}}\int_{\mathbb{R}^{N}}h\left( x\right) \left\vert \phi \right\vert
^{q}dx \\
&\leq &\frac{b}{2}\Vert \phi \Vert ^{2}-\int_{\mathbb{R}^{N}}\lim_{t%
\rightarrow +\infty }\frac{G(x,t\phi )}{t^{2}\phi ^{2}}\phi ^{2}dx \\
&=&\frac{b}{2}\Vert \phi \Vert ^{2}-\frac{1}{2}\int_{\mathbb{R}%
^{N}}p_{2}\left( x\right) \phi ^{2}dx \\
&=&\frac{1}{2}\left( b\Vert \phi \Vert ^{2}-1\right) <0,
\end{eqnarray*}%
where $I_{0,h}=I_{a,h}$ for $a=0.$ So, if $I_{0,h}(t\phi )\rightarrow
-\infty $ as $t\rightarrow +\infty $, then there exists $e\in H^{1}(\mathbb{R%
}^{N})$ with $\Vert e\Vert >\rho $ such that $I_{0,h}(e)<0$. Since $%
I_{a,h}(e)\rightarrow I_{0,h}(e)$ as $a\rightarrow 0^{+}$, there exists $%
a^{\ast }>0$ such that $I_{a,h}(e)<0$ for all $a\in \left( 0,a^{\ast
}\right) $ and the lemma is proved.
\end{proof}

\begin{lemma}
\label{lem5}Suppose that conditions ${(D_{1})}-\left( D_{3}\right) $ and $%
\left( D_{6}\right) $ hold. Let $h\in L^{2/\left( 2-q\right) }\left( \mathbb{%
R}^{N}\right) $ and $a^{\ast }>0$ be as in Lemma \ref{lem2}. Then for every $%
a\in \left( 0,a^{\ast }\right) $ there exists $D_{a}<0$ such that%
\begin{equation*}
D_{a}\leq \widetilde{\theta }_{h}:=\inf \left\{ I_{a,h}(u):u\in H^{1}(%
\mathbb{R}^{N})\right\} <0.
\end{equation*}%
Furthermore, there exists $R_{h}>0$ such that $I_{a,h}(u)>0$ for all $u\in
H^{1}(\mathbb{R}^{N})$ with $\left\Vert u\right\Vert \geq R_{h}$, and
\begin{equation*}
\inf \left\{ I_{a,h}(u):u\in H^{1}(\mathbb{R}^{N})\right\} =\inf \left\{
I_{a,h}(u):u\in B_{R_{h}}\right\} <0.
\end{equation*}%
In particular,
\begin{equation*}
\inf \left\{ I_{a,0}(u):u\in H^{1}(\mathbb{R}^{N})\right\} =\inf \left\{
I_{a,0}(u):u\in B_{R_{h}}\right\} <0,
\end{equation*}%
where $I_{a,0}(u)=I_{a,h}(u)$ for $h\equiv 0.$
\end{lemma}

\begin{proof}
By conditions ${(D_{1})}-\left( D_{3}\right) $ and $\left( D_{6}\right) ,$
and noticing that $\lim_{s\rightarrow +\infty }\frac{g\left( x,s\right) }{%
s^{r-1}}=0$ uniformly in $x\in \mathbb{R}^{N}$ for any fixed $2<r<\min
\left\{ 2+\delta _{0},2^{\ast }\right\} $ ($2^{\ast }=\infty $ if $N=1,2$
and $2^{\ast }=\frac{2N}{N-2}$ if $N\geq 3$), it is easy to see that for
every $\epsilon _{0}>0$ with $\left\vert p_{1}^{+}\right\vert _{\infty
}+\epsilon _{0}<b$, there is a $C_{0}=C\left( \epsilon _{0},r,g\right) >0$
such that%
\begin{eqnarray*}
I_{a,h}(u) &=&\frac{a}{2}\widehat{m}\left( \Vert u\Vert ^{2}\right) +\frac{b%
}{2}\Vert u\Vert ^{2}-\int_{\mathbb{R}^{N}}G(x,u)dx-\int_{\mathbb{R}%
^{N}}h\left\vert u^{+}\right\vert ^{q}dx \\
&\geq &\frac{ad_{0}}{2+2\delta _{0}}\Vert u\Vert ^{2+2\delta _{0}}+\frac{%
b-\left\vert p_{1}^{+}\right\vert _{\infty }-\epsilon _{0}}{2}\Vert u\Vert
^{2}-\frac{C_{0}}{rS_{r}^{r}}\Vert u\Vert ^{r}-S_{2}^{-q}\left\vert
h^{+}\right\vert _{L^{2/\left( 2-q\right) }}\left\Vert u\right\Vert ^{q} \\
&\geq &\frac{ad_{0}}{2+2\delta _{0}}\Vert u\Vert ^{2+2\delta _{0}}-\frac{%
C_{0}}{rS_{r}^{r}}\Vert u\Vert ^{r}-S_{2}^{-q}\left\vert h^{+}\right\vert
_{L^{2/\left( 2-q\right) }}\left\Vert u\right\Vert ^{q}
\end{eqnarray*}%
for all $u\in H^{1}(\mathbb{R}^{N})$, where we have used $\left( \ref{3.4}%
\right) $ and the condition $\left( D_{6}\right) .$ Then%
\begin{equation*}
I_{a,h}(u)\geq \frac{ad_{0}}{2+2\delta _{0}}\left( \Vert u\Vert ^{2+2\delta
_{0}}-\frac{\left( 2+2\delta _{0}\right) \left\vert h^{+}\right\vert
_{L^{2/\left( 2-q\right) }}}{ad_{0}S_{2}^{q}}\left\Vert u\right\Vert ^{q}-%
\frac{\left( 2+2\delta _{0}\right) C_{0}}{ard_{0}S_{r}^{r}}\Vert u\Vert
^{r}\right) .
\end{equation*}%
We now apply to Lemma \ref{m2} above with%
\begin{equation*}
\overline{A}=\frac{\left( 2+2\delta _{0}\right) \left\vert h^{+}\right\vert
_{L^{2/\left( 2-q\right) }}}{ad_{0}S_{2}^{q}}>0\text{ and }\overline{B}=%
\frac{\left( 2+2\delta _{0}\right) C_{0}}{ard_{0}S_{r}^{r}}>0,
\end{equation*}%
then for every $a\in \left( 0,a^{\ast }\right) $, there exist $D_{a}<0$ and $%
R_{h}>t_{\overline{B}}=\left[ \overline{B}\left( r-q\right) /\left(
k-q\right) \right] ^{1/\left( k-r\right) }$ such that
\begin{equation*}
I_{a,h}(u)\geq D_{a}\text{ for all }u\in H^{1}(\mathbb{R}^{N}),
\end{equation*}%
and
\begin{equation}
I_{a,h}(u)>0\text{ for all }u\in H^{1}(\mathbb{R}^{N})\text{ with }%
\left\Vert u\right\Vert \geq R_{h},  \label{3.11}
\end{equation}%
which implies that%
\begin{equation*}
\inf \left\{ I_{a,h}(u):u\in H^{1}(\mathbb{R}^{N})\right\} \geq D_{a}.
\end{equation*}%
Moreover, by Lemma \ref{lem2}, for any $a\in \left( 0,a^{\ast }\right) $%
\begin{equation}
\inf \left\{ I_{a,h}(u):u\in H^{1}(\mathbb{R}^{N})\right\} <0.  \label{3.12}
\end{equation}%
Therefore, combining $\left( \ref{3.11}\right) $ and $\left( \ref{3.12}%
\right) ,$ we see that%
\begin{equation*}
\inf \left\{ I_{a,h}(u):u\in H^{1}(\mathbb{R}^{N})\right\} =\inf \left\{
I_{a,h}(u):u\in \overline{B}_{R_{h}}\right\} <0.
\end{equation*}%
This completes the proof.
\end{proof}

By Theorem \ref{t2} and Lemmas \ref{lem1}, \ref{lem2}, we obtain that there
is a sequence $\{u_{n}\}\subset H^{1}(\mathbb{R}^{N})$ such that
\begin{equation}
I_{a,h}(u_{n})\rightarrow \alpha >0\quad \text{and}\quad (1+\Vert u_{n}\Vert
)\Vert I_{a,h}^{\prime }(u_{n})\Vert _{H^{-1}\left( \mathbb{R}^{N}\right)
}\rightarrow 0,\quad \text{as}\ n\rightarrow \infty .  \label{3.5}
\end{equation}

\begin{lemma}
\label{lem3}Suppose that conditions $\left( D_{1}\right) ,\left(
D_{2}\right) $ and $\left( D_{5}\right) $ hold. Then $\{u_{n}\}$ defined in $%
(\ref{3.5})$ is bounded in $H^{1}(\mathbb{R}^{N}).$
\end{lemma}

\begin{proof}
By contradiction, let $\Vert u_{n}\Vert \rightarrow +\infty $ as $%
n\rightarrow \infty $. Define $w_{n}:=\frac{u_{n}}{\Vert u_{n}\Vert }.$
Clearly, $w_{n}$ is bounded in $H^{1}(\mathbb{R}^{N})$ and there is $w\in
H^{1}(\mathbb{R}^{N})$ such that, up to a subsequence,
\begin{equation*}
w_{n}\rightharpoonup w\text{ weakly in }H^{1}(\mathbb{R}^{N})\text{ and }%
w_{n}\rightarrow w\text{ strongly in}\ L_{loc}^{2}(\mathbb{R}^{N})\text{ as}%
\ n\rightarrow \infty .
\end{equation*}%
It follows from $(\ref{3.5})$ that%
\begin{equation*}
\frac{\langle I_{a,h}^{\prime }(u_{n}),u_{n}\rangle }{\Vert u_{n}\Vert ^{2}}%
=o(1),
\end{equation*}%
that is,
\begin{equation}
o(1)=am\left( \Vert u_{n}\Vert ^{2}\right) +b-\int_{\mathbb{R}^{N}}\frac{%
g(x,u_{n})}{u_{n}}w_{n}^{2}dx-\frac{\int_{\mathbb{R}^{N}}h\left( x\right)
\left\vert w_{n}^{+}\right\vert ^{q}dx}{\Vert u_{n}\Vert ^{2-q}},
\label{3.6}
\end{equation}%
where $o(1)$ denotes a quantity which goes to zero as $n\rightarrow \infty $
and $\delta _{0}$ is as in the condition $\left( D_{5}\right) .$ By
conditions $\left( D_{1}\right) $ and $\left( D_{2}\right) $ and $h\in
L^{2/\left( 2-q\right) }\left( \mathbb{R}^{N}\right) $, there exist $%
C_{1},C_{2}>0$ such that
\begin{equation*}
\frac{g(x,s)}{s}\leq C_{1}\text{ for all}\ s\in \mathbb{R}
\end{equation*}%
and%
\begin{equation*}
\int_{\mathbb{R}^{N}}h\left( x\right) \left\vert w_{n}^{+}\right\vert
^{q}dx\leq C_{2}\left\vert h^{+}\right\vert _{L^{2/\left( 2-q\right) }},
\end{equation*}%
which implies that $\int_{\mathbb{R}^{N}}\frac{g(x,u_{n})}{u_{n}}w_{n}^{2}dx$
and $\int_{\mathbb{R}^{N}}h\left( x\right) \left\vert w_{n}^{+}\right\vert
^{q}dx$ are bounded in $H^{1}(\mathbb{R}^{N}).$ Moreover, by the condition $%
\left( D_{5}\right) ,$%
\begin{equation*}
m\left( \Vert u_{n}\Vert ^{2}\right) \rightarrow +\infty \text{ as }%
n\rightarrow \infty .
\end{equation*}%
So the above equation $\left( \ref{3.6}\right) $ is a contradiction.
Therefore, $\{u_{n}\}$ is bounded in $H^{1}(\mathbb{R}^{N}).$ This completes
the proof.
\end{proof}

\section{Proof of Theorem \protect\ref{t1-1}}

To prove that the Cerami sequence $\{u_{n}\}$ in $(\ref{3.5})$ converges to
a nonzero critical point of $I_{a,h}$, the following compactness lemma is
useful.

\begin{lemma}
\label{lem4-1}Suppose that conditions $\left( D_{1}\right) -\left(
D_{4}\right) $ hold and $\left\{ u_{n}\right\} $ is a $\left( PS\right)
_{\beta }$-sequence for $I_{a,h}$ in $H^{1}(\mathbb{R}^{N}),$ that is%
\begin{equation*}
I_{a,h}(u_{n})\rightarrow \beta \quad \text{and}\quad \Vert I_{a,h}^{\prime
}(u_{n})\Vert _{H^{-1}\left( \mathbb{R}^{N}\right) }\rightarrow 0,\quad
\text{as}\ n\rightarrow \infty .
\end{equation*}%
If $\left\{ u_{n}\right\} $ is a bounded sequence in $H^{1}(\mathbb{R}^{N})$%
, then for any $\epsilon >0$, there exist $R(\epsilon )>R_{0}$ and $%
n(\epsilon )>0$ such that $\int_{|x|\geq R}\left( |\nabla
u_{n}|^{2}+u_{n}^{2}\right) dx\leq \epsilon $ for all $n\geq n(\epsilon )$
and $R\geq R(\epsilon ).$
\end{lemma}

\begin{proof}
Let $\xi _{R}:\mathbb{R}^{3}\rightarrow \lbrack 0,1]$ be a smooth function
such that
\begin{equation}
\xi _{R}(x)=\left\{
\begin{array}{cc}
0, & \quad 0\leq |x|\leq R, \\
1, & \quad |x|\geq 2R,%
\end{array}%
\right.  \label{3.20}
\end{equation}%
and for some constant $C>0$ (independent of $R$),
\begin{equation}
|\nabla \xi _{R}(x)|\leq \frac{C}{R},\quad \text{for all}\ x\in \mathbb{R}%
^{3}.  \label{3.21}
\end{equation}%
Then, for all $n\in \mathbb{N}$ and $R\geq R_{0}$, we have
\begin{eqnarray*}
\int_{\mathbb{R}^{N}}|\nabla (u_{n}\xi _{R})|^{2}dx &=&\int_{\mathbb{R}%
^{N}}|\nabla u_{n}|^{2}\xi _{R}^{2}dx+\int_{\mathbb{R}^{N}}|u_{n}|^{2}|%
\nabla \xi _{R}|^{2}dx \\
&\leq &\int_{R<|x|<2R}|\nabla u_{n}|^{2}dx+\int_{|x|>2R}|\nabla u_{n}|^{2}dx+%
\frac{C^{2}}{R^{2}}\int_{\mathbb{R}^{N}}u_{n}^{2}dx \\
&\leq &(2+\frac{C^{2}}{R^{2}})\Vert u_{n}\Vert ^{2}\leq (2+\frac{C^{2}}{%
R_{0}^{2}})\Vert u_{n}\Vert ^{2}.
\end{eqnarray*}%
This implies that
\begin{equation}
\Vert u_{n}\xi _{R}\Vert \leq (3+\frac{C^{2}}{R_{0}^{2}})^{\frac{1}{2}}\Vert
u_{n}\Vert ,\text{ for all }n\in \mathbb{N}\text{ and }R\geq R_{0}.
\label{3.22}
\end{equation}%
Since $\left\{ u_{n}\right\} $ is a bounded sequence in $H^{1}(\mathbb{R}%
^{N})$, it follows that $\Vert I_{a,h}^{\prime }(u_{n})\Vert _{H^{-1}}\Vert
u_{n}\Vert \rightarrow 0$ as $n\rightarrow \infty $. So, for any $\epsilon
>0 $, there exists $n(\epsilon )>0$ such that
\begin{equation}
\Vert I_{a,h}^{\prime }(u_{n})\Vert _{H^{-1}}\left\Vert u_{n}\right\Vert
\leq \epsilon (3+\frac{C^{2}}{R_{0}^{2}})^{-\frac{1}{2}},\text{ for all }%
n>n(\epsilon ).  \label{3.23}
\end{equation}%
Hence, it follows from $(\ref{3.22})$ and $(\ref{3.23})$ that
\begin{equation}
|\langle I_{a,h}^{\prime }(u_{n}),u_{n}\xi _{R}\rangle |\leq \Vert
I_{a,h}^{\prime }(u_{n})\Vert _{H^{-1}}\Vert u_{n}\xi _{R}\Vert \leq
\epsilon ,  \label{3.24}
\end{equation}%
for all $n>n(\epsilon )$ and $R>R_{0}.$ Note that%
\begin{eqnarray}
&&\langle I_{a,h}^{\prime }(u_{n}),u_{n}\xi _{R}\rangle  \notag \\
&=&am\left( \Vert u_{n}\Vert ^{2}\right) \left( \int_{\mathbb{R}^{N}}|\nabla
u_{n}|^{2}\xi _{R}dx+\int_{\mathbb{R}^{N}}u_{n}^{2}\xi _{R}dx+\int_{\mathbb{R%
}^{N}}u_{n}\nabla u_{n}\nabla \xi _{R}dx\right)  \notag \\
&&+b\left( \int_{\mathbb{R}^{N}}|\nabla u_{n}|^{2}\xi _{R}dx+\int_{\mathbb{R}%
^{N}}u_{n}^{2}\xi _{R}dx+\int_{\mathbb{R}^{N}}u_{n}\nabla u_{n}\nabla \xi
_{R}dx\right)  \notag \\
&&-\int_{\mathbb{R}^{N}}g(x,u_{n})u_{n}\xi _{R}dx-\int_{\mathbb{R}%
^{N}}h\left\vert u_{n}^{+}\right\vert ^{q}\xi _{R}dx  \notag \\
&\geq &b\left( \int_{\mathbb{R}^{N}}|\nabla u_{n}|^{2}\xi _{R}dx+\int_{%
\mathbb{R}^{N}}u_{n}^{2}\xi _{R}dx\right) +(am\left( \Vert u_{n}\Vert
^{2}\right) +b)\int_{\mathbb{R}^{N}}u_{n}\nabla u_{n}\nabla \xi _{R}dx
\notag \\
&&-\int_{\mathbb{R}^{N}}g(x,u_{n})u_{n}\xi _{R}dx-\int_{\mathbb{R}%
^{N}}h\left\vert u_{n}^{+}\right\vert ^{q}\xi _{R}dx.  \label{3.25}
\end{eqnarray}%
For any $\epsilon >0,$ there exists $R_{1}\left( \epsilon \right) >R_{0}$
such that
\begin{equation}
\frac{1}{R^{2}}\leq \frac{4\epsilon ^{2}}{C^{2}}\text{ for all}\ R\geq
R(\epsilon ).  \label{3.26}
\end{equation}%
By $(\ref{3.26})$ and the Young inequality, we get, for all $n\in \mathbb{N}$
and $R\geq R(\epsilon )$,%
\begin{eqnarray}
\int_{\mathbb{R}^{N}}|u_{n}\nabla u_{n}\nabla \xi _{R}|dx &\leq &\epsilon
\int_{\mathbb{R}^{N}}|\nabla u_{n}|^{2}dx+\frac{1}{4\epsilon }\int_{|x|\leq
2R}|u_{n}|^{2}\frac{C^{2}}{R^{2}}dx  \notag \\
&\leq &\epsilon \int_{\mathbb{R}^{N}}|\nabla u_{n}|^{2}dx+\epsilon
\int_{|x|\leq 2R}|u_{n}|^{2}dx  \notag \\
&\leq &\epsilon \Vert u_{n}\Vert ^{2}.  \label{3.27}
\end{eqnarray}%
Moreover, since $h\in L^{2/\left( 2-q\right) }\left( \mathbb{R}^{N}\right) ,$
using the Egorov theorem and the H\"{o}lder inequality, there exists $%
R_{2}\left( \epsilon \right) >R_{0}$ such that for all $n\in \mathbb{N}$ and
$R\geq R_{2}(\epsilon )$, we have%
\begin{equation}
\int_{\mathbb{R}^{N}}h\left\vert u_{n}^{+}\right\vert ^{q}\xi _{R}<\epsilon .
\label{3.18}
\end{equation}%
Take $R(\epsilon )=\max \left\{ R_{1}(\epsilon ),R_{2}(\epsilon )\right\} .$
Now we consider two cases. \newline
\textbf{Case (i):} $b\geq 1.$ By $(D_{1}),(D_{4})$ and $(\ref{3.21})$, there
exists $\eta _{1}\in (0,1)$ such that, for all $n\in \mathbb{N}$ and $R\geq
R_{0},$
\begin{equation*}
\int_{\mathbb{R}^{N}}|g(x,u_{n})u_{n}\xi _{R}|dx\leq \eta _{1}\int_{\mathbb{R%
}^{N}}u_{n}^{2}\xi _{R}dx.
\end{equation*}%
Using this, together with $(\ref{3.25}),\left( \ref{3.27}\right) $ and $%
\left( \ref{3.18}\right) $, for all $n\in \mathbb{N}$ and $R\geq R(\epsilon
)\geq R_{0}$, we see that%
\begin{eqnarray}
&&\langle I_{a,h}^{\prime }(u_{n}),u_{n}\xi _{R}\rangle  \notag \\
&\geq &b\left( \int_{\mathbb{R}^{N}}|\nabla u_{n}|^{2}\xi _{R}dx+\int_{%
\mathbb{R}^{N}}u_{n}^{2}\xi _{R}dx\right)  \notag \\
&&+(am\left( \Vert u_{n}\Vert ^{2}\right) +b)\int_{\mathbb{R}%
^{N}}u_{n}\nabla u_{n}\nabla \xi _{R}dx-\int_{\mathbb{R}^{N}}g(x,u_{n})u_{n}%
\xi _{R}dx-\int_{\mathbb{R}^{N}}h\left\vert u_{n}^{+}\right\vert ^{q}\xi _{R}
\notag \\
&\geq &\int_{\mathbb{R}^{N}}|\nabla u_{n}|^{2}\xi _{R}dx+\int_{\mathbb{R}%
^{N}}u_{n}^{2}\xi _{R}dx+(am\left( \Vert u_{n}\Vert ^{2}\right) +1)\int_{%
\mathbb{R}^{N}}u_{n}\nabla u_{n}\nabla \xi _{R}dx  \notag \\
&&-\int_{\mathbb{R}^{N}}g(x,u_{n})u_{n}\xi _{R}dx-\int_{\mathbb{R}%
^{N}}h\left\vert u_{n}^{+}\right\vert ^{q}\xi _{R}  \notag \\
&\geq &\int_{\mathbb{R}^{N}}|\nabla u_{n}|^{2}\xi _{R}dx+(1-\eta _{1})\int_{%
\mathbb{R}^{N}}u_{n}^{2}\xi _{R}dx-\epsilon \left[ \left( am\left( \Vert
u_{n}\Vert ^{2}\right) +1\right) \Vert u_{n}\Vert ^{2}+1\right] .
\label{3.29}
\end{eqnarray}%
Since $\{u_{n}\}$ is bounded in $H^{1}(\mathbb{R}^{N})$, it follows from $(%
\ref{3.24})$ and $(\ref{3.29})$ that there exists $C_{2}>0$ such that for
all $n\geq n(\epsilon )$ and $R\geq R(\epsilon ),$%
\begin{equation}
\int_{\mathbb{R}^{N}}|\nabla u_{n}|^{2}\xi _{R}dx+(1-\eta _{1})\int_{\mathbb{%
R}^{N}}u_{n}^{2}\xi _{R}dx\leq C_{2}\epsilon .  \label{3.30}
\end{equation}%
From $\eta _{1}\in (0,1)$ and $(\ref{3.20})$, it is easy to see that $(\ref%
{3.30})$ implies the final conclusion.\newline
\textbf{Case (ii)} $0<b<1.$ By $(D_{1}),(D_{4})$ and $(\ref{3.21})$, there
exists $\eta _{2}\in (0,1)$ such that, for all $n\in \mathbb{N}$ and $R\geq
R_{0},$%
\begin{equation*}
\int_{\mathbb{R}^{N}}|g(x,u_{n})u_{n}\xi _{R}|dx\leq b\eta _{2}\int_{\mathbb{%
R}^{N}}u_{n}^{2}\xi _{R}dx.
\end{equation*}%
Similar to the proof of Case (i), we have%
\begin{eqnarray*}
&&\langle I_{a,h}^{\prime }(u_{n}),u_{n}\xi _{R}\rangle \\
&\geq &b\left( \int_{\mathbb{R}^{N}}|\nabla u_{n}|^{2}\xi _{R}dx+\int_{%
\mathbb{R}^{N}}u_{n}^{2}\xi _{R}dx\right) \\
&&+\left[am\left( \Vert u_{n}\Vert ^{2}\right) +b\right]\int_{\mathbb{R}%
^{N}}u_{n}\nabla u_{n}\nabla \xi _{R}dx-\int_{\mathbb{R}^{N}}g(x,u_{n})u_{n}%
\xi _{R}dx-\int_{\mathbb{R}^{N}}h\left\vert u_{n}^{+}\right\vert ^{q}\xi _{R}
\\
&\geq &b\int_{\mathbb{R}^{N}}|\nabla u_{n}|^{2}\xi _{R}dx+b(1-\eta
_{2})\int_{\mathbb{R}^{N}}u_{n}^{2}\xi _{R}dx \\
&&-\epsilon \left[ (am\left( \Vert u_{n}\Vert ^{2}\right) +b)\Vert
u_{n}\Vert ^{2}+1\right] ,
\end{eqnarray*}%
and there exists $C_{3}>0$ such that for all $n\geq n(\epsilon )$ and $R\geq
R(\epsilon )$,
\begin{equation}
\int_{\mathbb{R}^{N}}|\nabla u_{n}|^{2}\xi _{R}dx+(1-\eta _{2})\int_{\mathbb{%
R}^{N}}u_{n}^{2}\xi _{R}dx\leq C_{3}\epsilon .  \label{3.33}
\end{equation}%
From $\eta _{2}\in (0,1)$ and $(\ref{3.20})$, it is easy to see that $(\ref%
{3.33})$ also implies the final conclusion.
\end{proof}

\begin{lemma}
\label{lem4}Suppose that conditions $\left( D_{1}\right) -\left(
D_{4}\right) $ hold. Let $\left\{ u_{n}\right\} $ be a sequence as in $%
\left( \ref{3.5}\right) .$ Then for any $\epsilon >0$, there exist $%
R(\epsilon )>R_{0}$ and $n(\epsilon )>0$ such that $\int_{|x|\geq R}\left(
|\nabla u_{n}|^{2}+u_{n}^{2}\right) dx\leq \epsilon $ for all $n\geq
n(\epsilon )$ and $R\geq R(\epsilon ).$
\end{lemma}

\begin{proof}
Clearly, $\left\{ u_{n}\right\} $ is a $\left( PS\right) _{\alpha }$%
-sequence for $I_{a,h}$ in $H^{1}(\mathbb{R}^{N}).$ Moreover, by Lemma \ref%
{lem3}, $\left\{ u_{n}\right\} $ is a bounded sequence in $H^{1}(\mathbb{R}%
^{N}).$ Thus, by Lemma \ref{lem4-1}, it is easy to see that this lemma holds.
\end{proof}

\begin{theorem}
\label{t3}Suppose that conditions $\left( D_{1}\right) -\left( D_{5}\right) $
hold and $h\in L^{2/\left( 2-q\right) }\left( \mathbb{R}^{N}\right) .$ Let ${%
a}^{\ast }>0$ be as in Lemma \ref{lem2}. Then for each $a\in \left(
0,a^{\ast }\right) ,$ $I_{a,h}$ has a nonzero critical point $u_{0}\in H^{1}(%
\mathbb{R}^{3})$ such that $I_{a,h}\left( u_{0}\right) =\alpha >0.$
\end{theorem}

\begin{proof}
By Lemma \ref{lem3}, the sequence $\{u_{n}\}$ in $(\ref{3.5})$ is bounded in
$H^{1}(\mathbb{R}^{N}).$ We may assume that, up to a subsequence,
\begin{eqnarray}
u_{n} &\rightharpoonup &u_{0}\text{ weakly in }H^{1}(\mathbb{R}^{N});  \notag
\\
u_{n} &\rightarrow &u_{0}\text{ strongly in }L_{loc}^{r}\left( \mathbb{R}%
^{N}\right) \text{ for }2\leq r<2^{\ast };  \label{3.19} \\
u_{n} &\rightarrow &u_{0}\text{ a.e. in }\mathbb{R}^{N}  \notag
\end{eqnarray}%
for some $u_{0}\in H^{1}(\mathbb{R}^{N}).$ Moreover, since $h\in L^{2/\left(
2-q\right) }\left( \mathbb{R}^{N}\right) ,$ using the Egorov theorem and the
H\"{o}lder inequality, we have%
\begin{equation}
\int_{\mathbb{R}^{N}}h\left\vert u_{n}^{+}\right\vert ^{q-2}u_{n}^{+}\left(
u_{n}^{+}-u_{0}^{+}\right) =o\left( 1\right) .  \label{3.36}
\end{equation}%
In order to prove our conclusion, it is now sufficient to show that $%
\left\Vert u_{n}\right\Vert \rightarrow \left\Vert u_{0}\right\Vert $ as $%
n\rightarrow \infty $. Note that, by $(\ref{3.5})$,%
\begin{eqnarray*}
\langle I_{a,h}^{\prime }(u_{n}),u_{n}\rangle &=&\left[ am\left( \Vert
u_{n}\Vert ^{2}\right) +b\right] \int_{\mathbb{R}^{N}}(|\nabla
u_{n}|^{2}+u_{n}^{2})dx \\
&&-\int_{\mathbb{R}^{N}}g(x,u_{n})u_{n}dx-\int_{\mathbb{R}%
^{N}}h(x)\left\vert u_{n}^{+}\right\vert ^{q}dx \\
&=&o(1),
\end{eqnarray*}%
and%
\begin{eqnarray*}
\left\langle I_{a,h}^{\prime }(u_{n}),u_{0}\right\rangle &=&\left[ am\left(
\Vert u_{n}\Vert ^{2}\right) +b\right] \int_{\mathbb{R}^{N}}(\nabla
u_{n}\nabla u_{0}+u_{n}u_{0})dx \\
&&-\int_{\mathbb{R}^{N}}g(x,u_{n})u_{0}dx-\int_{\mathbb{R}%
^{N}}h(x)\left\vert u_{n}^{+}\right\vert ^{q-2}u_{n}^{+}u_{0}^{+}dx \\
&=&o(1).
\end{eqnarray*}%
Since $u_{n}\rightharpoonup u_{0}$ weakly in $H^{1}(\mathbb{R}^{N})$, it
follows that
\begin{equation}
\int_{\mathbb{R}^{N}}(\nabla u_{n}\nabla u_{0}+u_{n}u_{0})dx=\int_{\mathbb{R}%
^{N}}(|\nabla u_{0}|^{2}+u_{0}^{2})dx+o(1).  \label{3.37}
\end{equation}%
So by $\left( \ref{3.36}\right) $ and $\left( \ref{3.37}\right) $, to show $%
\left\Vert u_{n}\right\Vert \rightarrow \left\Vert u_{0}\right\Vert $ is
equivalent to prove that
\begin{equation}
\int_{\mathbb{R}^{N}}g(x,u_{n})u_{n}dx=\int_{\mathbb{R}%
^{N}}g(x,u_{n})u_{0}dx+o(1).  \label{3.34}
\end{equation}%
Indeed, for any $\epsilon >0$, by the condition $\left( D_{4}\right) ,$ the H%
\"{o}lder inequality and Lemma \ref{lem4}, for $n$ large enough, one has%
\begin{eqnarray*}
&&\int_{|x|\geq R(\epsilon )}g(x,u_{n})u_{n}dx-\int_{|x|\geq R(\epsilon
)}g(x,u_{n})u_{0}dx \\
&\leq &\int_{|x|\geq R(\epsilon )}\left\vert g(x,u_{n})\right\vert
\left\vert u_{n}-u_{0}\right\vert dx \\
&\leq &\min \left\{ 1,b\right\} \int_{|x|\geq R(\epsilon
)}|u_{n}||u_{n}-u_{0}|dx \\
&\leq &\min \left\{ 1,b\right\} \left( \int_{|x|\geq R(\epsilon
)}|u_{n}|^{2}dx\right) ^{\frac{1}{2}}\left( \int_{|x|\geq R(\epsilon
)}|u_{n}-u_{0}|^{2}dx\right) ^{\frac{1}{2}} \\
&\leq &\min \left\{ 1,b\right\} \epsilon .
\end{eqnarray*}%
Combining this and $\left( \ref{3.19}\right) $, the equation $(\ref{3.34})$
holds. This completes the proof.
\end{proof}

\bigskip

\textbf{Now we give the proof of Theorem \ref{t1-1}:} $\left( i\right) $
Theorem \ref{t3} shows the conclusion.\newline
$\left( ii\right) $ By Lemma \ref{lem5} and the Ekeland variational
principle, there exists a minimizing sequence $\{u_{n}\}\subset B_{R_{h}}$
such that
\begin{equation*}
I_{a,0}(u_{n})\rightarrow \widetilde{\theta }_{0}\text{ and }I_{a,0}^{\prime
}(u_{n})\rightarrow 0\text{ as }n\rightarrow \infty ,
\end{equation*}%
where
\begin{equation}
\widetilde{\theta }_{0}:=\inf \{I_{a,0}(u):u\in \overline{B}_{R_{h}}\}<0,
\label{4.3}
\end{equation}%
and $I_{a,0}=I_{a,h}$ for $h\equiv 0.$ Since $\left\{ u_{n}\right\} $ is a
bounded sequence, similar argument to the proof of Theorem \ref{t3}, there
exist a subsequence $\left\{ u_{n}\right\} $ and $u_{0}^{-}\in B_{R_{h}}$
such that $u_{n}\rightarrow u_{0}^{-}$ strongly in $H^{1}(\mathbb{R}^{N}),$
which implies that $I_{a,0}^{\prime }(u_{0}^{-})=0$ and $I_{a,0}(u_{0}^{-})=%
\widetilde{\theta }_{0}<0.$ This yields a nontrivial solution $u_{0}^{-}$ of
Equation $\left( K_{a,h}\right) .$ Moreover, by Theorem \ref{t3}, there
exists a nontrivial solution $u_{0}^{+}$ of Equation $\left( K_{a,h}\right) $
such that $I_{a,0}\left( u_{0}^{+}\right) =\alpha >0.$ Therefore,
\begin{equation*}
I_{a,0}(u_{0}^{-})=\widetilde{\theta }_{0}<0<\alpha =I_{a,0}\left(
u_{0}^{+}\right)
\end{equation*}%
which implies that $u_{0}^{-}\neq u_{0}^{+}.$ This completes the proof.

\section{Proof of Theorem \protect\ref{t1-2}}

For $h\in L^{2/\left( 2-q\right) }\left( \mathbb{R}^{N}\right) ,$ we define%
\begin{equation*}
\widetilde{\theta }_{h}:=\inf \{I_{a,h}(u):u\in H^{1}(\mathbb{R}^{N})\}.
\end{equation*}%
Then we have the following result.

\begin{theorem}
\label{t4}Suppose that conditions $\left( D_{1}\right) -\left( D_{4}\right) $
and $\left( D_{6}\right) $ hold. Then there exist $\overline{a}_{0},%
\overline{\Lambda }_{0}>0$ such that for every $a\in \left( 0,\overline{a}%
_{0}\right) $ and $h\in L^{2/\left( 2-q\right) }\left( \mathbb{R}^{N}\right)
$ with $h\geq 0$ and $0<\left\vert h\right\vert _{L^{2/\left( 2-q\right) }}<%
\overline{\Lambda }_{0},$ we have%
\begin{equation*}
\widetilde{\theta }_{h}\leq \widetilde{\theta }_{0}<\theta _{h}<0,
\end{equation*}%
where $\theta _{h}$ is as in $\left( \ref{4.2}\right) $ and $\widetilde{%
\theta }_{0}$ is as in $\left( \ref{4.3}\right) $. Furthermore, there exists
a nontrivial solution $u_{h,2}^{-}$ of Equation $\left( K_{a,h}\right) $
such that $I_{a,h}(u_{h,2}^{-})=\widetilde{\theta }_{h}.$
\end{theorem}

\begin{proof}
If $h\equiv 0,$ using Theorem \ref{t1-1} $\left( ii\right) $, then there
exists a nontrivial solution $u_{0}^{-}$ of Equation $\left( K_{a,h}\right) $
such that%
\begin{equation*}
I_{a,0}(u_{0}^{-})=\widetilde{\theta }_{0}=\inf \{I_{a,0}(u):u\in H^{1}(%
\mathbb{R}^{N})\}<0,
\end{equation*}%
where $I_{a,0}=I_{a,h}$ for $h\equiv 0.$ Assume that $h\in L^{2/\left(
2-q\right) }\left( \mathbb{R}^{N}\right) \backslash \left\{ 0\right\} $ with
$h\geq 0.$ Then
\begin{equation*}
I_{a,h}(u_{0}^{-})=\widetilde{\theta }_{0}-\int_{\mathbb{R}%
^{N}}h(x)\left\vert u_{0}^{-}\right\vert ^{q}dx<0.
\end{equation*}%
By Lemma \ref{lem5}, there exists $R_{\ast }>0$ such that $I_{a,h}(u)\geq 0$
for all $u\in H^{1}(\mathbb{R}^{N})$ with $\left\Vert u\right\Vert \geq
R_{\ast },$ and
\begin{eqnarray*}
\widetilde{\theta }_{h} &=&\inf \left\{ I_{a,h}(u):u\in H^{1}(\mathbb{R}%
^{N})\right\} \\
&=&\inf \left\{ I_{a,h}(u):u\in \overline{B}_{R_{\ast }}\right\} \leq
\widetilde{\theta }_{0}<0.
\end{eqnarray*}%
Moreover, by the Ekeland variational principle, there exists a minimizing
sequence $\{u_{n}\}\subset B_{R_{\ast }}$ such that
\begin{equation*}
I_{a,h}(u_{n})\rightarrow \widetilde{\theta }_{h}\text{ and }I_{a,h}^{\prime
}(u_{n})\rightarrow 0\text{ as }n\rightarrow \infty .
\end{equation*}%
Since $\left\{ u_{n}\right\} $ is a bounded sequence, similar argument to
the proof of Theorem \ref{t3}, there exist a subsequence $\left\{
u_{n}\right\} $ and $u_{h,2}^{-}\in B_{R_{\ast }}$ such that $%
u_{n}\rightarrow u_{h,2}^{-}$ strongly in $H^{1}(\mathbb{R}^{N}).$ This
implies that $I_{a,h}^{\prime }(u_{h,2}^{-})=0$ and $I_{a,h}(u_{h,2}^{-})=%
\widetilde{\theta }_{h}.$ This shows that $u_{h,2}^{-}$ is a nontrivial
solution of Equation $\left( K_{a,h}\right) .$

Next, we show that $\widetilde{\theta }_{0}<\theta _{h}.$ By Lemma \ref{lem1}%
, if $h\equiv 0,$ then there exists $\rho _{0}>\rho $ such that $%
I_{a,0}(u)>0 $ for all $u\in B_{\rho _{0}}\backslash \left\{ 0\right\} $ and
$\inf_{u\in B_{\rho _{0}}}I_{a,0}\left( u\right) =0.$ Thus, using $\left( %
\ref{4.2}\right) ,$ we can conclude that $\theta _{h}\rightarrow 0$ as $%
\left\vert h\right\vert _{L^{2/\left( 2-q\right) }}\rightarrow 0.$ Thus,
there exists $\overline{\Lambda }_{0}>0$ such that for every $h\in
L^{2/\left( 2-q\right) }\left( \mathbb{R}^{N}\right) $ with $h\geq 0$ and $%
0<\left\vert h\right\vert _{L^{2/\left( 2-q\right) }}<\overline{\Lambda }%
_{0},$ we have $\widetilde{\theta }_{0}<\theta _{h}.$ This completes the
proof.
\end{proof}

\bigskip

\textbf{Now we give the proof of Theorem \ref{t1-2}:} $\left( i\right) $
Since $h\in L^{2/\left( 2-q\right) }(\mathbb{R})\backslash \left\{ 0\right\}
$ with $0<\left\vert h^{+}\right\vert _{L^{2/\left( 2-q\right) }}<\Lambda
_{0},$ we can choose a function $\phi \in H^{1}\left( \mathbb{R}^{N}\right)
\backslash \left\{ 0\right\} $ such that
\begin{equation*}
\int_{\mathbb{R}^{N}}h(x)\left\vert \phi ^{+}\right\vert ^{q}dx>0.
\end{equation*}%
For $t>0,$ we have
\begin{eqnarray}
I_{a,h}(t\phi ) &=&\frac{a}{2}\widehat{m}\left( \Vert t\phi \Vert
^{2}\right) +\frac{bt^{2}}{2}\Vert \phi \Vert ^{2}-\int_{\mathbb{R}%
^{N}}G(x,t\phi )dx-t^{q}\int_{\mathbb{R}^{N}}h(x)\left\vert \phi
^{+}\right\vert ^{q}dx  \notag \\
&\leq &\frac{a}{2}\widehat{m}\left( \Vert t\phi \Vert ^{2}\right) +\frac{%
t^{2}}{2}\Vert \phi \Vert ^{2}-\frac{t^{2}}{2}\int_{\mathbb{R}^{N}}p_{1}\phi
^{2}dx-t^{q}\int_{\mathbb{R}^{N}}h(x)\left\vert \phi ^{+}\right\vert ^{q}dx
\notag \\
&<&0  \label{4.1}
\end{eqnarray}%
for $t>0$ small enough. Moreover, by Lemma \ref{lem1}, there exists $%
\widetilde{R}_{h}\leq \rho $ such that
\begin{equation*}
I_{a,h}\left( u\right) \geq 0\text{ for all }u\text{ with }\left\Vert
u\right\Vert =\widetilde{R}_{h}.
\end{equation*}%
Hence,
\begin{equation}
\theta _{h}:=\inf \{I_{a,h}(u):u\in \overline{B}_{\widetilde{R}_{h}}\}<0.
\label{4.2}
\end{equation}%
By the Ekeland variational principle, there exists a minimizing sequence $%
\{u_{n}\}\subset B_{\widetilde{R}_{h}}$ such that $I_{a,h}(u_{n})\rightarrow
\theta _{h}$ and $I_{a,h}^{\prime }(u_{n})\rightarrow 0$ as $n\rightarrow
\infty .$ Since $\left\{ u_{n}\right\} $ is a bounded sequence, similar
argument to the proof of Theorem \ref{t3}, there exists a subsequence $%
\left\{ u_{n}\right\} $ and $u_{h,1}^{-}\in B_{\widetilde{R}_{h}}$ such that
$u_{n}\rightarrow u_{h,1}^{-}$ strongly in $H^{1}(\mathbb{R}^{N}),$ which
implies that $I_{a,h}^{\prime }(u_{h,1}^{-})=0$ and $I_{a,h}(u_{h,1}^{-})=%
\theta _{h}<0.$ This shows that $u_{h,1}^{-}$ is a nontrivial solution of
Equation $\left( K_{a,h}\right) .$\newline
$\left( ii\right) $ By part $\left( i\right) $ and Theorem \ref{t3}, there
exist two nontrivial solutions $u_{h,1}^{-}$ and $u_{h}^{+}$ of Equation $%
\left( K_{a,h}\right) $ such that
\begin{equation*}
I_{a,h}\left( u_{h,1}^{-}\right) =\theta _{h}<0<\alpha =I_{a,h}\left(
u_{h}^{+}\right) ,
\end{equation*}%
where $\theta _{h}$ is as in $\left( \ref{4.2}\right) .$ This shows that $%
u_{h,1}^{-}\neq u_{h}^{+}.$\newline
$\left( iii\right) $ By part $\left( ii\right) $ and Theorem \ref{t4},
Equation $\left( K_{a,h}\right) $ has three nontrivial solutions $%
u_{h,1}^{-},u_{h,2}^{-}$ and $u_{h}^{+}$ with%
\begin{equation*}
\widetilde{\theta }_{h}=I_{a,h}\left( u_{h,2}^{-}\right) <I_{a,h}\left(
u_{h,1}^{-}\right) =\theta _{h}<0<\alpha =I_{a,h}\left( u_{h}^{+}\right) ,
\end{equation*}%
which implies that $u_{h,1}^{-},u_{h,2}^{-}$ and $u_{h}^{+}$ are different.
This completes the proof.

\section{Proof of Theorem \protect\ref{t1-3}}

\textbf{Now we give the proof of Theorem \ref{t1-3}:} Let $u$ be a
nontrivial solution of Equation $\left( K_{a,h}\right) .$ Then we have
\begin{eqnarray*}
0 &=&\langle I_{a,b}^{\prime }(u),u\rangle \\
&=&am\left( \left\Vert u\right\Vert ^{2}\right) \left\Vert u\right\Vert
^{2}+b\left\Vert u\right\Vert ^{2}-\int_{\mathbb{R}^{N}}g(x,u)udx-\int_{%
\mathbb{R}^{N}}h\left\vert u^{+}\right\vert ^{q}dx.
\end{eqnarray*}%
The condition $(D_{7})$ implies that
\begin{equation*}
\int_{\mathbb{R}^{N}}g(x,u)udx\leq \int_{\mathbb{R}^{N}}p_{2}u^{2}dx.
\end{equation*}%
Thus, by $\left( \ref{2}\right) $ and the condition $\left( D_{8}\right) ,$%
\begin{eqnarray*}
0 &\geq &am\left( \left\Vert u\right\Vert ^{2}\right) \left\Vert
u\right\Vert ^{2}+b\left\Vert u\right\Vert ^{2}-\int_{\mathbb{R}%
^{N}}p_{2}u^{2}dx-\int_{\mathbb{R}^{N}}h\left\vert u^{+}\right\vert ^{q}dx \\
&\geq &am\left( \left\Vert u\right\Vert ^{2}\right) \left\Vert u\right\Vert
^{2}-\frac{1-b\mu ^{\ast }}{\mu ^{\ast }}\left\Vert u\right\Vert ^{2}-\frac{%
\left\vert h\right\vert _{L^{2/\left( 2-q\right) }}}{S_{2}^{q}}\Vert u\Vert
^{q} \\
&>&0,
\end{eqnarray*}%
which is a contradiction. Therefore, Equation $\left( K_{a,h}\right) $ does
not admits any nontrivial solution. This completes the proof.

\end{document}